\documentclass[12pt]{article}
\usepackage{amsmath,amsthm,amssymb,graphicx}
\usepackage[margin=1.25in]{geometry}
\newtheorem{corollary}{Corollary}
\newtheorem{remark}{Remark}
\newtheorem{lemma}{Lemma}
\newtheorem{theorem}{Theorem}[section]

\usepackage{hyperref}
\usepackage{authblk}

\title{\textbf{Bernstein-type Inequalities Preserved by Modified Smirnov Operator}}
\author{{Ishfaq Ahmad Wani$^{1}$, Abdul Liman$^{2}$}}
\affil{Department of Mathematics, National Institute of Technology, Srinagar, India, 190006.\\
Email Id: ishfa\textunderscore 2022phamth006@nitsri.ac.in, aliman@nitsri.ac.in}
\date{}
\begin{document}
	
	\maketitle
	
	\begin{abstract}
		In this paper, we consider a modified version of Smirnov operator and obtain some Bernstein-type inequalities preserved by this operator. In particular, we prove some compact generalizations of the well-known inequalities of Bernstein, Erd\"{o}s and Lax, Ankeny and Rivlin and others.
	\end{abstract}

	\section{\protect \bigskip Introduction}
	Let $\mathbb{P}_{n}$ denote the class of polynomials $P(z):=\sum_{j=0}^{n}a_jz^j$ in $\mathbb{C}$ of degree at most  $n \in \mathbb{N}$. Let $\mathbb{D}$ be the open unit disk $\left\{z \in \mathbb{C};|z|<1\right\}$, such that $\overline{\mathbb{D}}$ is the closure of $\mathbb{D}$ and $B(\mathbb{D})$ denotes its boundary.\\
	Let $P \in \mathbb{P}_{n}$, then 
	\begin{eqnarray}\label{1}
		\max_{z \in B(\mathbb{D})}|P^{\prime}(z)|\le n \max_{z \in B(\mathbb{D})}|P(z)|
	\end{eqnarray}
	and
	\begin{eqnarray}\label{2}
		\max_{z \in B(\mathbb{D})}|P(Rz)|\le R^{n} \max_{z \in B(\mathbb{D})}|P(z)|.
	\end{eqnarray}
	Inequality (\ref{1}) is a well-known theorem of Bernstein \cite{SB}. The inequality (\ref{2}) is a simple deduction from the maximum modulus principle. In both the inequalities, the equality holds for $P(z)=\alpha{z^{n}}, \alpha\ne 0$.\\
	If we restrict to a class of polynomials having no zeros in $\mathbb{D}$, the inequalities (\ref{1}) and (\ref{2}) can be sharpened. In fact, if $P(z)\ne 0$ in $\mathbb{D}$, then \\
	\begin{eqnarray}\label{3}
	\max_{z \in B(\mathbb{D})}|P^{\prime}(z)|\le \frac{n}{2} \max_{z \in B(\mathbb{D})}|P(z)|
	\end{eqnarray}
	and for $R>1,$
	\begin{eqnarray}\label{4}
	\max_{z \in B(\mathbb{D})}|P(Rz)|\le \frac{R^{n}+1}{2} \max_{z \in B(\mathbb{D})}|P(z)|.
	\end{eqnarray}
	Inequality (\ref{3}) was proved by Erd\"{o}s and Lax \cite{EL8}, whereas Ankeny and Rivlin \cite{AR1} used (\ref{3}) to prove (\ref{4}). These inequalities were further improved by Aziz and Dawood \cite{AD2}, where under the same hypothesis, it was proved that
	\begin{eqnarray}
	\max_{z \in B(\mathbb{D})}|P^{\prime}(z)|\le \frac{n}{2}\left\{\max_{z \in B(\mathbb{D})}|P(z)|- \min_{z \in B(\mathbb{D})}|P(z)|\right\}
	\end{eqnarray}
	and for $R>1$
	\begin{equation}\label{6}
	\max_{z\in B(\mathbb{D})}|P(Rz)|\le \left\{\frac{R^n+1}{2}\right\}\max_{z \in B(\mathbb{D})}|P(z)|-\left\{\frac{R^n-1}{2}\right\} \min_{z \in B(\mathbb{D})}|P(z)|.
	\end{equation}
	The equality in (\ref{3})-(\ref{6}) holds for the polynomials of the form $P(z)=\alpha {z}^n+\beta$, with $|\alpha|=|\beta|$.\\
	
	In 1930 Bernstein \cite{BB4} also proved the following result:
	\begin{theorem} \label{Theorem 1.1}
	Let $P(z)$ be a polynomial in $\mathbb{P}_{n}$ having all zeros in $\overline{\mathbb{D}}$ and $p(z)$ be a polynomial of degree not exceeding that of $P(z)$. If $|p(z)|\le|P(z)|$ on $B(\mathbb{D})$, then \\
	$$|p^{\prime}(z)|\le|P^{\prime}(z)| ~~for~~ z\in{\mathbb{C}\setminus{\mathbb{D}}}.$$ 	The equality holds only if $p=e^{i\gamma}P,\gamma \in \mathbb{R}$.
	\end{theorem}
	For $z\in{\mathbb{C}\setminus{\mathbb{D}}}$, denoting by $\Omega_{|z|}$ the image of the disc $\left\{t \in \mathbb{C};|t|<|z|\right\}$ under the mapping $\phi(t)=\frac{t}{1+t}$, Smirnov \cite{SM9} as a generalization of Theorem \ref{Theorem 1.1} proved the following:
	\begin{theorem}\label{Theorem 1.2}
	Let $p$ and $P$ be polynomails possessing conditions as in Theorem \ref{Theorem 1.1}, then for $z\in{\mathbb{C}\setminus{\mathbb{D}}}$
	\begin{equation}\label{7}
		|\mathbb{S}_{\alpha}[p](z)|\le |\mathbb{S}_{\alpha}[P](z)|
	\end{equation}
	for all $\alpha \in \overline{\Omega}_{|z|}$, with $\mathbb{S}_{\alpha}[p](z):=zp^{\prime}(z)-n\alpha p(z)$, where $\alpha$ is a constant.
\end{theorem}
For $\alpha \in \overline{\Omega}_{|z|}$ in inequality (\ref{7}), the equality holds at a point $z \in {\mathbb{C}\setminus{\mathbb{D}}}$ only if $p=e^{i\gamma}P,~\gamma \in \mathbb{R}$. We note that for fixed $z \in {\mathbb{C}\setminus{\mathbb{D}}}$, Inequality (\ref{7}) can be replaced by (see for reference $\cite{GS6,EGVV}$)
	\begin{eqnarray*}
		\left|zp^{\prime}(z)-n\frac{az}{1+az}p(z)\right|\le \left|zP^{\prime}(z)-n\frac{az}{1+az}P(z)\right|,
	\end{eqnarray*}
	where $a$ is arbitrary from $\overline{\mathbb{D}}$.
	Equivalently for $z\in{\mathbb{C}\setminus{\mathbb{D}}}$
	\begin{eqnarray*}
		|\tilde{\mathbb{S}}_{a}[p](z)|\le |\tilde{\mathbb{S}}_{a}[P](z)|,
	\end{eqnarray*}
	where $\tilde{\mathbb{S}}_{a}[p](z)= (1+az)p^{\prime}(z)-na p(z)$ is known as modified Smirnov operator.
	The modified Smirnov operator $\tilde{\mathbb{S}}_{a}$ is more preferred in a sense than Smirnov operator $\mathbb{S}_{\alpha}$, because the parameter $a$ of $\tilde{\mathbb{S}}_{a}$ does not depend on $z$ unlike parameter $\alpha$ of $\mathbb{S}_{\alpha}$.\\\\
		Marden \cite{MM} introduced a differential operator $\mathbb{B}: \mathbb{P}_{n} \to \mathbb{P}_{n}$ of $mth$ order. This operator
	carries a polynomial $p \in \mathbb{P}_{n}$ into
	\begin{eqnarray*}
		\mathbb{B}[p](z)=\lambda_{0} p(z)+ \lambda_{1} \frac{nz}{2}p^{\prime}(z)+...+\lambda_{m} \left(\frac{nz}{2}\right)^m p^{m}(z),
	\end{eqnarray*}
	where $\lambda_{0}, \lambda_{1},...,\lambda_{m}$ are constants such that
	\begin{eqnarray}\label{8}
		u(z)=\lambda_{0} +\binom{n}{1} \lambda_{1}z +...+ \binom{n}{m} \lambda_{m} z^m\ne0,\quad for \quad Re(z)>\frac{n}{4}.
	\end{eqnarray}
	Rahman and Schmeisser \cite{QR} considered the Marden operator for $m=2$ and showed
	that this operator preserves the inequalities between polynomials and accordingly proved the following:
	\begin{theorem} \label{Theorem 1.3}
		Let $p$ and $P$ be polynomials possessing conditions as in Theorem \ref{Theorem 1.1}, Then
		\begin{eqnarray}\label{9}
			|\mathbb{B}[p](z)| \le |\mathbb{B}[P](z)|\quad for \quad z \in \mathbb{C}\setminus{\mathbb{D}},
		\end{eqnarray}
		where the constants $\lambda_{0}, \lambda_{1},...,\lambda_{m}$ possess condition (\ref{8}). For  $z\in{\mathbb{C}\setminus{\mathbb{D}}}$ in (\ref{9}), the equality holds if and only if $p(z) = \gamma z^n,~\gamma \ne0$.
	\end{theorem}
	A variety of key papers concerning the $\mathbb{B}$-operator have appeared in the literature \cite{IQ, WL}.
	
	In order to compare the Smirnov operator $\mathbb{S}_{\alpha}[p](z):=zp^{\prime}(z)-n\alpha p(z)$ and the
	Rahman’s operator (with $\lambda_{2} = 0$) $\mathbb{B}[p](z)=\lambda_{0} p(z)+ \lambda_{1} \frac{nz}{2}p^{\prime}(z)$, we require $\alpha \in \overline{\Omega}_{|z|}$ in
	inequality (\ref{7}) and in inequality (\ref{9}) the root of the polynomial $u(z) = \lambda_{0} + n \lambda_{1} z$ should lie in the half-plane $Re(z) \le \frac{n}{4}$, that is
	\begin{eqnarray*}
		Re\left(-\frac{\lambda_{0}}{\lambda_{1}n}\right) \le \frac{n}{4}.
	\end{eqnarray*}
	Compare the sets of parameters in Theorem \ref{Theorem 1.2} and Theorem \ref{Theorem 1.3}, we see that in Theorem \ref{Theorem 1.2}, this set(coefficient near $- p(z))$ is $\mathcal{A}= \left\{ n \alpha : \alpha \in \Omega_{|z|}\right\}$ and in Theorem \ref{Theorem 1.3}, the set of such coefficient near $ - p(z)$ is
	
	$$\mathcal{B}=\left\{-\frac{2 \lambda_{0}}{\lambda_{1} n}:Re\left( -\frac{ \lambda_{0}}{\lambda_{1} n}\right) \le \frac{n}{4} \right\}=\left\{ t : Re(t)\le \frac{n}{2}\right\}.$$
	Consider the differential inequalities from Theorem \ref{Theorem 1.2} and Theorem \ref{Theorem 1.3} for $z \in B(\mathbb{D})$, we have $\mathcal{A}=\mathcal{B}$. But for $z\in{\mathbb{C}\setminus{\mathbb{D}}}$ we have $\mathcal{B}\subset\mathcal{A}$. In other words in Theorem \ref{Theorem 1.2}
	and Theorem \ref{Theorem 1.3} formally the same inequality was obtained but for different set of
	parameters. Moreover, the set of parameters in Theorem \ref{Theorem 1.2} is essentially wider than that of Theorem \ref{Theorem 1.3}.
	Consequently,
	\begin{eqnarray}
		\mathbb{B}[p](z) = \lambda_{1} \frac{n}{2}{S}_{\alpha}[p](z).
	\end{eqnarray}
	These facts were first observed by Ganenkova and Starkov \cite{GS6}.\\
	In this paper, we prove some more general results concerning the modified Smirnov operator preserving inequalities between polynomials, which in turn yields compact generalizations of some well-known polynomial inequalities.
	\section{Auxiliary Results}
	Before writing our main results, we prove the following lemmas which are required for their proofs.
	\begin{lemma}\label{lemma 1}
		Let $P\in \mathbb{P}_{n}$, and has all zeros in $\overline{\mathbb{D}}$. Let $a \in B(\mathbb{D})$ be not the exceptional value for $P$, then all the zeros of $\tilde{\mathbb{S}}_{a}[P]$ lie in $\overline{\mathbb{D}}$.
	\end{lemma}
The above lemma is due to Genenkova and starkov $\cite{GS6}$.
	Also, the next lemma is due to Aziz \cite{AA}.
	\begin{lemma} \label{lemma 2}
		If \( p(z) \) is a polynomial of degree \( n \) having all its zeros in \( |z| \leq k \), where \( k \geq 0 \), then for every \( R \geq r \) and \( rR \geq k^2 \),
		\[
		|p(Rz)| \geq \left( \frac{R + k}{r + k} \right)^n |p(rz)| \quad \text{for} \quad z \in B(\mathbb{D}).\]
	\end{lemma}
	\begin{lemma}\label{lemma 3}
		If $p\in \mathbb{P}_{n}$ with $|p(z)|<\mathbb{M}$ for $z \in B(\mathbb{D})$, Then 
		\begin{equation*}
			|\tilde{\mathbb{S}}_{a}[p](z)|\le \mathbb{M}|\tilde{\mathbb{S}}_{a}[z^n]| ~~~for~~~ z\in{\mathbb{C}\setminus{\mathbb{D}}}.
		\end{equation*}
	\end{lemma}
	\begin{proof}
		Since  $|p(z)|<\mathbb{M}$ for $z \in B(\mathbb{D})$. If $\lambda$ is a complex number with $|\lambda|>1$. Then 
		\begin{eqnarray*}
			|p(z)|<|\lambda\mathbb{M}z^n|~~~ for~~~ z \in B(\mathbb{D}).
		\end{eqnarray*}
		Since $\lambda\mathbb{M}z^n$ has all zeros in  $\overline{\mathbb{D}}$, therefore by Rouche's theorem all zeros of $p(z)-\lambda\mathbb{M}z^n$ lie in $\overline{\mathbb{D}}$. Hence by Lemma \ref{lemma 1}, all zeros of $\tilde{\mathbb{S}}_{a}[p(z)-\lambda{\mathbb{M}}{z^n}]$ lie in $\overline{\mathbb{D}}$. Since $\tilde{\mathbb{S}}_{a}$ is linear, it follows that ${\tilde{\mathbb{S}}_{a}[p](z)}-{\tilde{\mathbb{S}}_{a}[\lambda\mathbb{M}z^n]}$ has all zeros in $\overline{\mathbb{D}}$.\\
		This gives
		\begin{equation} \label{11}
			|\tilde{\mathbb{S}}_{a}[p](z)|\le \mathbb{M}|\tilde{\mathbb{S}}_{a}[z^n]| ~~~for~~~ z\in{\mathbb{C}\setminus{\mathbb{D}}}.
		\end{equation}
		Because if this is not true, then there exist some $z_0 \in {\mathbb{C}\setminus{\mathbb{D}}}$ such that 
		\begin{equation*}
			|\tilde{\mathbb{S}}_{a}[p]({z_0})|> \mathbb{M}|\tilde{\mathbb{S}}_{a}[{z_0}^n]|.
		\end{equation*}
		Choosing $\lambda=\frac{\tilde{\mathbb{S}}_{a}[p]({z_0})}{\mathbb{M}|\tilde{\mathbb{S}}_{a}[{z_0}^n]},$ so that $|\lambda|>1$. With this choice of $\lambda$, we get a contradiction and hence inequality (\ref{11}) is true.
	\end{proof}
	The next two Lemmas are given by Shah and Fatima \cite{WMS}.
	\begin{lemma} \label{lemma 4}
		If  $p\in \mathbb{P}_{n}$, then for $ z \in \mathbb{C} \setminus \mathbb{D}$
		\begin{eqnarray}
			|\tilde{S}_{a}[p](z)| + |\tilde{S}_{a}[g](z)| \le \left\{| \tilde{S}_{a}[E_n](z)|+n|a|\right\} \max_{z \in B(\mathbb{D})}|p(z)|,
		\end{eqnarray}
		where $g(z)=z^n \overline{p(\frac{1}{\bar{z}})}$.
	\end{lemma}
	\begin{lemma}\label{Lemma 5}
		Let $p(z)$ and $P(z)$ be two polynomials such that $\deg p(z) \leq \deg P(z) = n$. If $P(z)$ has all zeros in $\mathbb{D}$ and $|p(z)| \leq |P(z)|$ for $z \in B(\mathbb{D})$, then for any complex number $\beta$ with $\beta \in \overline{\mathbb{D}}$ and $R \geq 1$, we have for $z \in B(\mathbb{D})$
		\begin{equation}
			\left| \tilde{S}_{a}[p](Rz) - \beta \tilde{S}_{a}[p](z) \right| \leq \left| \tilde{S}_{a}[P](Rz) - \beta \tilde{S}_{a}[P](z) \right|.
		\end{equation}
		The result is sharp and equality holds if $a \in \overline{\mathbb{D}}$ is not the exceptional value for the polynomial $p(z) = e^{i\gamma}P(z)$, where $\gamma \in \mathbb{R}$ and $P(z)$ is any polynomial having all the zeros in $\overline{\mathbb{D}}$ and strict inequality holds for $z \in \mathbb{D}$, unless $p(z) = e^{i\gamma} P(z)$, $\gamma \in \mathbb{R}$.
	\end{lemma}
	We now prove the following result which is a compact generalization of inequalities (\ref{1}) and (\ref{2}).
	\section{Main Results}
	\begin{theorem}\label{Theorem 3.1}
		If \( p(z) \) is a polynomial of degree \( n \), then, for every real or complex number $\beta, |\beta| \le 1$ and $R\ge1$,
		\begin{eqnarray} \label{14}
			\big| \tilde{S}_a[p](Rz)- \beta \tilde{S}_a[p](z)\big| \le \big|R^n-\beta \big| \big|\tilde{S}_a[E_n](z)\big|\max_{z \in B(\mathbb{D})}|p(z)| \quad for \quad  z \in\mathbb{C} \setminus \mathbb{D}.
		\end{eqnarray}
		Equivalently for $R>1$
		\begin{align}\label{15}
		\notag	&\big|(1 + az)[R P^{\prime}(Rz) - \beta P^{\prime}
			(z)]-na[P(Rz) - \beta P(z)]\big|~~~~~~~~~~~~~\\ &~~~~~~~~~~~ \qquad \qquad
			 \le n |R^n-\beta||z|^{n-1} {\max_{z\in B(\mathbb{D})}} |p(z)|  \quad for \quad  z \in\mathbb{C} \setminus \mathbb{D},
		\end{align}
		where $E_n(z)=z^n$. The result is sharp and holds for $p(z) = \gamma z^n,~\gamma \ne0$.
	\end{theorem}
	\begin{corollary}
		 For $\beta=0,~a=0$ and $R=1$, the inequality (\ref{15}) reduces to  
		\begin{eqnarray*}
			|P^{\prime}(z)| \le n |z|^{n-1}\max_{z \in B(\mathbb{D})}|P(z)| \quad for \quad  z \in\mathbb{C} \setminus \mathbb{D}
		\end{eqnarray*}
		which in particular gives inequality (\ref{1}). The equality holds for $p(z) = \gamma z^n,~\gamma \ne0$.
	\end{corollary}
	\begin{remark}
		For $\beta=0$ and $R>1$, the inequality (\ref{14}) reduces to a result due to Fatima and Shah \cite{WM}.
	\end{remark} 
	
	\begin{theorem} \label{Theorem 3.2}
		Let $P \in  \mathbb{P}_{n}$ and $Q(z)=z^n \overline{p(\frac{1}{\bar{z}})}$, then for every real or complex $\beta$ with $|\beta|\le1$ and $R>1$, 
		\begin{align} \label{16}
	\notag	&	|\tilde{S}_a[P](Rz) - \beta \tilde{S}_a[P](z)|+|\tilde{S}_a[Q](Rz) - \beta \tilde{S}_a[Q](z)| \\
			 &\qquad\qquad\qquad\le \left\{|R^n-\beta| \tilde{S}_a[E_n](z)+n|1- \beta||a|\right\} \max_{z \in B(\mathbb{D})}|p(z)|  \quad for \quad z \in \mathbb{C} \setminus \mathbb{D}.
		\end{align}
			\text{Equivalently}
		\begin{align} \label{17}
			\notag~~	&\bigg|(1 + az)[R P^{\prime}(Rz) - \beta P^{\prime}
			(z)]-na[P(Rz) - \beta P(z)]\\ &  \notag + (1 + az)[R Q^{\prime}(Rz) - \beta Q^{\prime}
			(z)]-na[Q(Rz) - \beta Q(z)]\bigg|\\ & \qquad \qquad
			\le \left\{n |R^n-\beta|  |z|^{n-1}+n|1- \beta||a|\right\} \max_{z \in B(\mathbb{D})}|p(z)|  \quad for \quad z \in \mathbb{C} \setminus \mathbb{D}.
		\end{align}
	\end{theorem}
	\begin{corollary}
		 If $P(z)$ is a polynomial of degree $n$, then for $\beta=1$, $a=0$  and $R \ge 1$ in inequality (\ref{17}), we get
		\begin{eqnarray*}
			\bigg|RP^{\prime}(Rz) -  P^{\prime}(z) \bigg|
			+\bigg|R Q^{\prime}(Rz) - Q^{\prime}(z) \bigg|
			\le	n (R^n-1)  |z|^{n-1}\max_{z \in B(\mathbb{D})}|P(z)|,
		\end{eqnarray*}
		where $E_n(z)=z^n$. The result is best possible and the equality holds for $p(z) = \gamma z^n,~\gamma \ne0$. Theorem \ref{Theorem 3.2} includes a result due
		to Rahman \cite{QR} as a special case.
	\end{corollary}
\begin{remark}
	 If we take $\beta=0$ and $R=1$ in (\ref{16}), then the inequality reduces to Lemma \ref{lemma 4}
	\begin{eqnarray*}
		\left| \tilde{S}_{a}[P](z) \right| 
		+ \left| \tilde{S}_{a}[Q](z) \right| 
		\leq 
		\bigg[ \tilde{S}_a[E_n](z) 
		+ n|a|\bigg] \max_{z \in B(\mathbb{D})}|P(z)|.
	\end{eqnarray*}
\end{remark}
	\begin{theorem} \label{Theorem 3.3}
		Let $P \in  \mathbb{P}_{n}$ such that  $P(z)$ is a polynomial of degree n which does not vanish in $\mathbb{D}$ and $Q(z)=z^n \overline{p(\frac{1}{\bar{z}})}$, then for every real or complex $\beta$ with $|\beta|\le1$ and $R>1$
		\begin{align} \label{18}
			&|\tilde{S}_a[P](Rz) - \beta \tilde{S}_a[P](z)| \notag \\ 
			& \quad \le \left\{\frac{|R^n-\beta| \tilde{S}_a[E_n](z)+n|1- \beta||a|}{2} \right\} \max_{z \in B(\mathbb{D})}|p(z)|  \quad for \quad z \in \mathbb{C} \setminus \mathbb{D}.
		\end{align}
			\text{Equivalently}
		\begin{align} \label{19}
			\notag	&|(1 + az)[R P^{\prime}(Rz) - \beta P^{\prime}
			(z)]-na[P(Rz) - \beta P(z)]\\ & \qquad \qquad
			\le \left\{ \frac{|R^n-\beta| n |z|^{n-1}+n|1- \beta||a|}{2}\right\} \max_{z \in B(\mathbb{D})}|p(z)|  \quad for \quad z \in \mathbb{C} \setminus \mathbb{D},
		\end{align}
		where $E_n(z)=z^n$. The result is best possible and the equality holds for $p(z) = \gamma z^n,~\gamma \ne0$.
	\end{theorem}
		\begin{corollary}
			 For $\beta=0,~a=0$ and $R=1$, the inequality (\ref{19}) reduces to  
			\begin{eqnarray*}
				|P^{\prime}(z)| \le \frac{n}{2} |z|^{n-1}\max_{z \in B(\mathbb{D})}|P(z)| \quad for \quad  z \in\mathbb{C} \setminus \mathbb{D}
			\end{eqnarray*}
			which in particular gives inequality (\ref{3}). The equality holds for $p(z) = \gamma z^n,~\gamma \ne0$.
		\end{corollary}
	\begin{remark}
		For $\beta=0,~R=1$, the inequality (\ref{18}) reduces to a result due to Shah and Fatima \cite{WMS}
		\begin{eqnarray*}
			|\tilde{S}_a[P](z)|  
			\le \frac{1}{2}\left\{\tilde{S}_a[E_n](z)+n|a|\right\} \max_{z \in B(\mathbb{D})}|p(z)|  \quad for \quad z \in \mathbb{C} \setminus \mathbb{D}.
		\end{eqnarray*}
	\end{remark}
	\section{Proofs of the theorems}
	\begin{proof}[Proof of Theorem \ref{Theorem 3.1}]
		For $R=1$, the result is trivial. Henceforth, we assume $R>1$.\\
		If $$\max_{z \in B(\mathbb{D})}|p(z)|=M,$$ then
		\begin{eqnarray*}
			|p(z)|<M\quad \text{for} \quad z \in B(\mathbb{D}).
		\end{eqnarray*}
		Equivalently for every $\lambda$ with $|\lambda|>1$, we have 
		\begin{eqnarray}
			|p(z)|<|M \lambda z^n| \quad\text{for} \quad  z \in  B(\mathbb{D}).
		\end{eqnarray}
		Therefore by Rouche’s theorem, it follows that all the zeros of $ F(z) =p(z)+M\lambda z^n$
		lie in $\mathbb{D}$.
		By Lemma \ref{lemma 1}, it follows that all the zeros of $\tilde{S}_a[F](z)$ lie in $\mathbb{D}$.\\
		So, all the zeros of $\tilde{S}_a[F](z)=\tilde{S}_a[p(z)+M\lambda z^n]$ lie in $\mathbb{D}$.
		Therefore, all the zeros of $\tilde{S}_a[p(z)]+M\lambda \tilde{S}_a[E_n](z)$ lie in $\mathbb{D}$, where $E_n(z)=z^n$.\\
		Now for any $\beta \in \mathbb{C}, ~|\beta|\le1$,
		by using the application of Lemma \ref{lemma 1}, it follows that all the zeros of 
		\begin{eqnarray*}
			\tilde{S}_a\left\{[F](Rz)- \beta [F](z)\right\}&=& (1+az)\left\{RF^{\prime}(Rz)-\beta F^{\prime}(z)\right\}-na\left\{F(Rz)-\beta F(z)\right\} \\
			&=& (1+az)RF^{\prime}(Rz) - na F(Rz) - \beta \left\{(1+az)F^{\prime}(z) -  (na) F(z)\right\}\\
			&=& \tilde{S}_a[F](Rz)- \beta \tilde{S}_a[F](z)
		\end{eqnarray*}
		lie in $\mathbb{D}$ for every $a$ such that $a \in B(\mathbb{D})$ is not the exceptional value of $F$.\\\\
		Since,
		$$\tilde{S}_a[F](z)= \tilde{S}_a[p](z)+\lambda  \tilde{S}_a[E_n](z) M $$
		and
		$$\tilde{S}_a[F](Rz)= \tilde{S}_a[p](Rz)+ \lambda R^n \tilde{S}_a[E_n](z)M.$$
		Therefore, all the zeros of 
		\begin{eqnarray*}
			\tilde{S}_a\left\{[F](Rz)- \beta [F](z)\right\} &=& \tilde{S}_a[F](Rz)- \beta \tilde{S}_a[F](z)\\
			&=&  \tilde{S}_a[p](Rz)+ \lambda  R^n \tilde{S}_a[E_n](z)M - \beta \left\{ \tilde{S}_a[p](z)+\lambda  \tilde{S}_a[E_n](z) M \right\}\\
			&=& \tilde{S}_a[p](Rz) - \beta \tilde{S}_a[p](z) + \lambda [R^n-|\beta|]\tilde{S}_a[E_n](z)M
		\end{eqnarray*}
		lie in $\mathbb{D}$ for $R>1,~|\lambda|>1$.\\
		This implies 
		\begin{eqnarray} \label{21}
			|\tilde{S}_a[p](Rz) - \beta \tilde{S}_a[p](z)| \le |R^n-\beta||\tilde{S}_a[E_n](z)|M \quad for \quad z \in \mathbb{C} \setminus \mathbb{D}, ~R>1.
		\end{eqnarray} 
		If inequality (\ref{21}) is not true, then there is some point $z_0 \in \mathbb{C} \setminus \mathbb{D}$ such that 
		\begin{eqnarray}
			|\tilde{S}_a[p](Rz_0) - \beta \tilde{S}_a[p](z_0)| > |R^n-\beta||\tilde{S}_a[E_n](z_0)|M \quad for \quad z_0 \in \mathbb{C} \setminus \mathbb{D}, ~R>1.
		\end{eqnarray}
		Take 
		\begin{eqnarray*}
			\lambda= -\frac{\tilde{S}_a[p](Rz_0) - \beta \tilde{S}_a[p](z_0)}{\left\{R^n-\beta\right\}\tilde{S}_a[E_n](z_0) M},
		\end{eqnarray*}
		such that $\lambda \in \mathbb{C} \setminus \mathbb{D}$ and with such choice of $\lambda$ we have for $z_0 \in \mathbb{C} \setminus \mathbb{D}$ 
		\begin{eqnarray*}
			\tilde{S}_a \left\{ [F](Rz_0)- \beta [F](z_0)\right\}=0
		\end{eqnarray*}
		which is a contradiction. Hence, we get
		\begin{eqnarray}
			|\tilde{S}_a[p](Rz) - \beta \tilde{S}_a[p](z)| \le |R^n-\beta||\tilde{S}_a[E_n](z)| \max_{z \in B(\mathbb{D})}|p(z)| \quad for \quad z \in \mathbb{C} \setminus \mathbb{D}.
		\end{eqnarray}
	\end{proof}
		\begin{proof}[Proof of Theorem \ref{Theorem 3.2}]
		Let $$\max_{z \in B(\mathbb{D})}|p(z)|=M$$
		then $|p(z)|\le M$ for $z \in \mathbb{D}$. Using Rouche's theorem, it follows that for every real or complex number $\alpha$ with $|\alpha|>1$, $F(z)=P(z)+ \alpha M$ does not vanish in $\mathbb{D}$.
		Using the Theorem \ref{Theorem 3.1}  and lemma \ref{Lemma 5}, on the polynomial $F(z)$, we get for every real or complex number $\beta$ with $|\beta|\le 1$ 
		\begin{align*}
			&\bigg|\tilde{S}_a[P(Rz) - \beta \tilde{S}_a[P](z) +n\alpha(1-\beta)|a|M\bigg| \\
			& \qquad\leq \bigg|\tilde{S}_a[Q](Rz) - \beta \tilde{S}_a[Q](z)+ \alpha(R^n-\beta)\tilde{S}_a[E_n](z)M\bigg| \quad for \quad z \in\mathbb{C} \setminus \mathbb{D},
		\end{align*}
		where $Q(z)=z^n \overline{p(\frac{1}{\bar{z}})}$.\\
		Choosing the argument of $\alpha$ in R.H.S of above inequality such that
		\begin{align*}
			&|\tilde{S}_a[Q](Rz) - \beta \tilde{S}_a[Q](z)+ \alpha(R^n-\beta)\tilde{S}_a[E_n](z)M|\\& =|\alpha||(R^n-\beta)|\tilde{S}_a[E_n](z)M - |\tilde{S}_a[Q](Rz) - \beta \tilde{S}_a[Q](z)|. 
		\end{align*}
		Therefore
		\begin{align*}
			&|\tilde{S}_a[P(Rz) - \beta \tilde{S}_a[P](z)| -n|\alpha||1-\beta||a|M \\
			& \le |\alpha||R^n-\beta|\tilde{S}_a[E_n](z)M - |\tilde{S}_a[Q](Rz) - \beta \tilde{S}_a[Q](z)|.
		\end{align*}
		This implies 
		\begin{align*}
			&|\tilde{S}_a[P(Rz) - \beta \tilde{S}_a[P](z)| +|\tilde{S}_a[Q](Rz) - \beta \tilde{S}_a[Q](z)| \\
			& \le |\alpha|\left\{ |(R^n-\beta)|\tilde{S}_a[E_n](z)+n|1- \beta||a|\right\}M.
		\end{align*}
		Now, letting $|\alpha|$ $\to1$, we get
		\begin{align}\label{24}
		\notag	& |\tilde{S}_a[P(Rz) - \beta \tilde{S}_a[P](z)| +|\tilde{S}_a[Q](Rz) - \beta \tilde{S}_a[Q](z)|\\
			& \le \left\{ |(R^n-\beta)|\tilde{S}_a[E_n](z)+n|1- \beta||a|\right\} \max_{z \in B(\mathbb{D})}|p(z)| \quad for \quad z \in\mathbb{C} \setminus \mathbb{D}.
		\end{align}
	\end{proof}
	\begin{proof}[Proof of Theorem \ref{Theorem 3.3}]
		Let $$\max_{z \in B(\mathbb{D})}|p(z)|=M,$$
		then for every real or complex number $\beta$ with $|\beta| \le 1$ and $R>1$, we have from inequality (\ref{24})
		\begin{align}
			\notag &|\tilde{S}_a[P(Rz) - \beta \tilde{S}_a[P](z)| +|\tilde{S}_a[Q](Rz) - \beta \tilde{S}_a[Q](z)|\\& \qquad
			\le \left\{ |(R^n-\beta)|\tilde{S}_a[E_n](z)+n|1- \beta||a|\right\} \max_{z \in B(\mathbb{D})}|p(z)| \quad for \quad z \in\mathbb{C} \setminus \mathbb{D}.
		\end{align}
		Also from the Lemma \ref{Lemma 5}, we have
		\begin{equation} \label{25}
			\left| \tilde{S}_{a}[P](Rz) - \beta \tilde{S}_{a}[P](z) \right| \leq \left| \tilde{S}_{a}[Q](Rz) - \beta \tilde{S}_{a}[Q](z) \right|.
		\end{equation}
		Adding $	\left| \tilde{S}_{a}[P](Rz) - \beta \tilde{S}_{a}[P](z) \right|$ on the both sides of the inequality (\ref{25}), we get
		\begin{eqnarray*}
			2 \left\{\left| \tilde{S}_{a}[P](Rz) - \beta \tilde{S}_{a}[P](z) \right| \right\}  \le |\tilde{S}_a[P(Rz) - \beta \tilde{S}_a[P](z)| +|\tilde{S}_a[Q](Rz) - \beta \tilde{S}_a[Q](z)|.
		\end{eqnarray*}
		Using the inequality (\ref{24}) in above inequality, we get
		\begin{eqnarray*}
			2 \left\{\left| \tilde{S}_{a}[P](Rz) - \beta \tilde{S}_{a}[P](z) \right| \right\} \le \left\{ |(R^n-\beta)|\tilde{S}_a[E_n](z)+n|1- \beta||a|\right\} \max_{z \in B(\mathbb{D})}|p(z)|.
		\end{eqnarray*}
		Therefore,
		\begin{eqnarray*}
			\left| \tilde{S}_{a}[P](Rz) - \beta \tilde{S}_{a}[P](z) \right|  \le\left\{ \frac{ |(R^n-\beta)|\tilde{S}_a[E_n](z)+n|1- \beta||a|}{2} \right\} \max_{z \in B(\mathbb{D})}|p(z)|
		\end{eqnarray*}
		for $ z \in\mathbb{C} \setminus \mathbb{D}$.
	\end{proof}
	\section{Declaration}
	\noindent\textbf{Conflicts of interest:} On behalf of authors, the corresponding author states that
	there is no conflict of interest.
		
\end{document}